\documentclass{article}
\usepackage{amsmath,amsfonts,amssymb,amsthm}
\usepackage{graphicx,subfigure,epstopdf}
\newtheorem{teo}{Theorem}

\newtheorem{defi}{Definition}

\newcommand{\lam}{$\lambda$}

\begin{document}

\title{Some regional control problems for population dynamics}
\author{Laura-Iulia Ani\c{t}a\thanks{Faculty of Physics, ``Alexandru Ioan Cuza'' University of Ia\c{s}i, Ia\c{s}i 700506, Romania, E-mail: lianita@uaic.ro}
\and Sebastian Ani\c{t}a\thanks{Faculty of Mathematics, ``Alexandru Ioan Cuza'' University of Ia\c{s}i, and ``Octav Mayer'' Institute of Mathematics of the Romanian Academy, Ia\c{s}i 700506, Romania, Corresponding author, E-mail: sanita@uaic.ro}
\and Vincenzo Capasso\thanks{ADAMSS (Centre for Advanced Applied Mathematical and Statistical Sciences),  and  Universit\'a degli Studi di Milano, 20133 Milano, Italy, E-mail: vincenzo.capasso@unimi.it}
\and Ana-Maria Mo\c sneagu\thanks{Faculty of Mathematics, ``Alexandru Ioan Cuza'' University of Ia\c{s}i, Ia\c{s}i 700506, Romania, E-mail: anamaria.mosneagu@uaic.ro}
}
\date{}
\maketitle
\begin{abstract}
This paper deals with some control problems related to structured population dynamics with diffusion.
Firstly, we investigate the regional control for an optimal harvesting problem (the control acts in a subregion $\omega$ of the whole domain $\Omega$). Using the necessary optimality conditions, for a fixed $\omega$, we get the structure of the harvesting effort which gives the maximum harvest; with this optimal effort we investigate the best choice of the subregion $\omega$ in order to maximize the harvest. We introduce an iterative numerical method to increase the total harvest at each iteration by changing the subregion where the effort acts. Numerical tests are used to illustrate the effectiveness of the theoretical results. We also consider the problem of eradication of an age-structured pest population dynamics with diffusion and logistic term, which is a zero-stabilization problem with constraints. We derive a necessary condition and a sufficient condition for zero-stabilizability. We formulate a related optimal control problem which takes into account the cost of intervention in the subregion $\omega$.
\end{abstract}
\noindent
\textbf{Keywords}  Optimal harvesting; population dynamics; diffusive models; regional control; numerical methods.
\section{Introduction}
An extensive literature was developed for the optimal harvesting
problems of population dynamics (e.g. \cite{AAV}--\cite{ACM},
\cite{AM}, \cite{VB2}, \cite{bv}, \cite{fl}, \cite{gm1}--\cite{h},
\cite{hy}, \cite{l2}--\cite{ms}, \cite{zwz2}, \cite{zzw}). In this
paper we firstly remind an optimal harvesting problem for a
spatially structured population with diffusion which has been
introduced in \cite{ACM}. For spatially structured harvesting
problems it has  usually been  taken  into consideration an effort
that acts in the whole habitat $\Omega$ (see  for example
\cite{anita}). Instead here we consider the case  in which  the
effort is localized in a suitably chosen subregion $\omega$ of
$\Omega.$  In addition to the problem of finding the magnitude of
of the control   to act on a given subdomain  $\omega,$  the most
important task will be to identify an optimal  subregion $\omega$,
where the control acts, in order to maximize the harvest. To this
aim, at first we have derived necessary optimality conditions for
the situation when the support of the control is fixed; as a  fall
out we  have  obtained  information concerning the structure of
the optimal control. Hence we  have taken into account this
structure  to investigate the optimal subregion $\omega$ where the
control is localized, by taking into account the cost paid for
harvesting in $\omega$. Here we have adapted some shape
optimization methods, based on the level set method. These results
have been previously presented in \cite{ACM}. In this paper we
consider also the problem of eradication of an age-structured pest
population dynamics with diffusion and logistic term. We consider
a related optimal control problem which can be again investigated
by means of  the level set method.

We consider the following population dynamics model with
diffusion. A  single population species is free to move in an
isolated habitat $\Omega \subset {\mathbb{R}^2}$, with $\Omega $ a
bounded domain with a sufficiently smooth boundary:
\begin{equation}
\label{a}
\left\{ \begin{array}{ll}
\partial_t y(x,t)- d\Delta y(x,t) = a(x) y(x,t) - \chi_{\omega }(x)u(x,t)y(x,t), &(x,t)\in Q_T \\
\partial _{\nu } y(x,t)=0, &(x,t)\in \Sigma_T\\
y(x,0)=y_0(x), &x\in \Omega, \\
\end{array}
\right.
\end{equation}
where $Q_T=\Omega \times (0,T)$, $\Sigma_T=\partial\Omega \times
(0,T)$, $T>0$, $y=y(x,t)$ is the population density at position
$x\in \overline{\Omega }$ and time $t\in [0,T]$, while $y_{0}(x)$
is the initial population density. Here $a(x)$ denotes the natural
growth rate of the population, and $d\in (0,+\infty)$ is the
diffusion coefficient. No-flux boundary conditions are considered.

In System  \eqref{a}, $u(x,t)$ represents the harvesting effort
(control), bounded and localized in the subdomain $\omega \subset
\Omega $ ($\chi _{\omega }$ is the characteristic function of
$\omega $). The term $\chi _{\omega }(x)u(x,t)y(x,t)$ represents
the rate of the harvested population at position $x\in\omega$ and
time $t\in[0,T]$.

The following hypotheses are considered:
\begin{description}
\item[{\bf (H1)}] $a\in L^{\infty}(\Omega)$;
\item[{\bf (H2)}] $y_{0}\in L^{\infty}(\Omega), \quad y_{0}(x)\geqslant 0 \quad \mbox{a.e.} \ x \in \Omega\,$ with $\,\|y_0\|_{L^{\infty}(\Omega)}>0$.
\end{description}
We consider a related optimal harvesting problem
\begin{equation}\label{OH}Maximize\int_{0}^{T}\int_{\omega} u(x,t)y^{u}(x,t)dx\ dt ,\end{equation}
subject to $u\in K_{\omega}$, where $K_{\omega}=\{ w\in L^{\infty}(\omega\times (0,T));\ 0\leq w(x,t)\leq L \ \ \mbox{\rm a.e. in } \omega \times (0,T)\}$. Here $L>0$ is a constant and $y^{u}$ is the solution to \eqref{a} corresponding to a harvesting effort $u\in K_{\omega}$.

The existence result of an optimal control for Problem (\ref{OH})
follows \cite{AAC} or \cite{AM}.
\begin{teo} Problem (\ref{OH}) admits at least one optimal control.
\end{teo}
We denote by $p$ the adjoint state, i.e. $p$ satisfies
\begin{equation}\left\{ \begin{array}{ll}
      \partial_t p(x,t) + d\Delta p(x,t) = -a(x) p(x,t)\\
       ~~~~~~~~~~~~~~~~~~~~~~~~~~~~~~+\chi_{\omega}(x)u^{*}(x,t)(1+p(x,t)) ,  \quad &(x,t)\in Q_T \\
\partial _{\nu } p(x,t)=0, \quad &(x,t)\in \Sigma_T \\
p (x,T)=0, \quad &x\in  \Omega, \\
           \end{array}
    \right. \label{f}\end{equation}
where $(u^{*},y^{u^{*}})$ is an optimal pair for (\ref{OH}). For the construction of the adjoint problems in optimal control theory we
refer to \cite{VB2}. Concerning  the first order necessary optimality conditions it can be proved the following result (as in \cite{AAC} and \cite{AM}):
\begin{teo} \label{th1}If $(u^{*},y^{u^{*}})$ is an optimal
pair for Problem (\ref{OH}) and if $p$ is the solution of Problem
(\ref{f}), then we have:
$$ u^{*}(x,t)=\left\{
\begin{array}{l l}
0, &\quad 1+p(x,t) < 0 \\
L, & \quad 1+p(x,t) > 0
\end{array} \right. \quad \mbox{a.e.} \ (x,t) \in \omega \times (0,T).$$
\end{teo}

In Section \ref{optimal_region}  we will treat the regional
harvesting problem as a shape optimization problem. We remind that
the geometry of a set $\omega$ can be characterized in terms of
its Minkowski functionals. There are three such functionals and
these are proportional to the area, the  perimeter and the
Euler-Poincar\'{e} characteristic. In this paper we control the
shape of $\omega$ as follows: by the length of the
boundary of  $\omega$, and  by the area of $\omega$.

We shall use the implicit interface representation to control the
shape of the 2D domain $\omega$. Therefore, the boundary of a
domain is defined as the isocontour of some  function $\varphi$
(see \cite{DZ} or \cite{OF}).  By using the level set method, we
introduce a  level set function $\varphi: \overline{\Omega}
\rightarrow \mathbb{R}$ such that $ \omega =  \{x \in  \Omega ; \
\varphi(x) > 0\ \mbox{\rm a.e.}\} $ and $\partial \omega =\{x\in
\Omega: \varphi(x) = 0 \ \mbox{\rm a.e.}\}$ (the boundary is
defined as the  zero level set of $\varphi$).  We will then
manipulate $\omega$ implicitly, through the function $\varphi$.
This function $\varphi$ is assumed to take positive values inside
the region delimited by the curve $\partial \omega$ and negative
values outside.

If  $\varphi$  is the implicit function of $\omega,$ in order  to
integrate over $\omega$ a function $f$ defined over the whole
$\Omega$ we may write $\int_{\Omega} f(x) H(\varphi (x)) dx,$
where  we have used the Heaviside function $H: {\mathbb{R}}
\rightarrow \{ 0,1\},$ such that
$$H(z)=\left\{
\begin{array}{ll}
1, \quad  \textnormal{if} \quad z \geq 0\\
0, \quad  \textnormal{if} \quad  z < 0. \\
\end{array}
\right. $$ If $\varphi $ is sufficiently smooth,  the directional
derivative of the Heaviside function in the normal direction at a
point  $x \in
\partial \omega$ is  given by $H'(\varphi(x))|\nabla \varphi(x)|,
$ and by using the usual Dirac  Delta $\delta$ on $\mathbb{R}$,
we have $\delta(\varphi(x))|\nabla \varphi(x)|$. If we need to
integrate over $\partial\omega$ a function $f$ defined over the
whole $\Omega$  we may write $\int_{\Omega} f(x) \delta (\varphi
(x)) |\nabla \varphi(x)| dx$.

We find the derivative of the optimal cost value with respect to
the implicit function $\varphi$ of  the subregion $\omega$. In
order to improve the region where the control acts we derive a
conceptual iterative algorithm based on these theoretical results.
We also present the numerical implementation of this  conceptual
algorithm and some numerical tests. Basically, the theoretical
results in Section 2 have been obtained in \cite{ACM}. Here we
give some additional details concerning the numerical scheme and
its implementation. Further  we present here some new numerical
tests.

In Section 3 we treat the problem of eradication of an
age-structured pest population with diffusion, which is a
zero-stabilization problem with constraints. We derive a necessary
condition and a sufficient condition of zero-stabilization. We
consider a related optimal control problem which takes into
account the cost paid by acting in the subregion $\omega$. We
formulate this optimal control problem by means of the  level set
method. The results in this section are new.

\section{An iterative method to localize an optimal subdomain $\omega $ where the control
acts}  \label{optimal_region}

Here we intend to use the level set method in order to obtain the
optimal subregion $\omega$ where the control is localized.
Consider $\varphi: \overline{\Omega} \rightarrow \mathbb{R}$  the
implicit function of $\omega$, the subregion of $\Omega$ where the
control acts.

We rewrite the optimal control problem \eqref{OH} such that will include  both the  magnitude
of the harvesting effort $u \in K_{\omega }$,
and the choice of the subdomain  $\omega $ with respect to its implicit function $\varphi$:
$$\underset{\varphi} {Maximize} \ \underset{u \in  K_{\omega}}  {Maximize}\left\{ \int_0^T\int_{\omega }u(x,t)y^u(x,t)dx \ dt\right.$$
$$-\alpha\; length(\partial\omega) - \beta\; area(\omega)\},$$
where $y^{u}$ is the solution to \eqref{a} corresponding to a harvesting effort $u\in K_{\omega}$ and $\alpha, \beta$ are positive constants. $\alpha\;\text{length}(\partial \omega )+\beta\;\text{area}(\omega )$ represents the cost paid to
harvest in the subregion $\omega $.

By using De Giorgi's formula for the length (perimeter) of a set
and assuming  that $\varphi$ is sufficiently smooth , the optimal
problem becomes
$$\underset{\varphi} {Maximize} \ \underset{u \in  K_{\omega}}  {Maximize}\left\{  \int_0^T\int_{\omega }u(x,t)y^u(x,t)dx \ dt\right.$$
$$\left.-\alpha\int_{\Omega} \delta (\varphi (x)) |\nabla \varphi(x)| dx-\beta\int_{\Omega} H(\varphi (x)) dx\right\}.$$

We have now two maximization problems: firstly, for a fixed $\varphi$ (and implicitly, $\omega$) we have to find the structure of the harvesting effort which gives the maximum harvest, as a function of $\varphi$ (or $\omega$); secondly, using this structure of the optimal control we investigate the optimal choice of the subregion $\omega$ with respect to its implicit function $\varphi$ in order to maximize the harvest.

For any arbitrary but fixed $\varphi$, we denote by $(u_{\varphi}^{*},y_{\varphi}^{*})$ an optimal pair for the harvesting problem \eqref{OH}. Now we have to investigate the following optimal control problem:
$$\underset{\varphi} {Maximize}\,\left\{\int_{0}^{T}\int_{\omega} u_{\varphi}^{*}(x,t)y_{\varphi}^{*}(x,t)dx\ dt\right.$$$$\left.-\alpha\int_{\Omega} \delta (\varphi (x)) |\nabla \varphi(x)| dx-\beta\int_{\Omega} H(\varphi (x)) dx\right\},$$
where $y_{\varphi}^{*}$ is the solution to
$$\left\{ \begin{array}{ll}
      \partial_t y (x,t)- d\Delta y(x,t) = a(x) y(x,t) -H(\varphi (x))u_{\varphi}^{*}(x,t)y(x,t),  &(x,t)\in Q_T \\
\partial _{\nu } y(x,t)=0, \quad &(x,t)\in \Sigma_T \\
y (x,0)=y_0(x), \quad &x\in  \Omega. \\
           \end{array}
    \right. $$
Note that $H(\varphi)$ represents  the characteristic function of
$\omega$.

Assume that the hypotheses (H1, H2) are satisfied. We denote by $p_{\varphi}$ the adjoint state. From Theorem \ref{th1}, the optimal control is given by
\begin{equation} u^{*}_{\varphi}(x,t)=\left\{
\begin{array}{l l}
0, &\quad 1+p_{\varphi}(x,t) < 0\\
L, & \quad 1+p_{\varphi}(x,t) > 0
\end{array} \right . \label{r}\end{equation} a.e. $(x,t)\in \omega \times (0,T)$, where $p_{\varphi}$ is the solution to (\ref{f}).\\

By multiplying (\ref{a}) by $p_{\varphi}$ and (\ref{f}) by
$y_{\varphi}^{*}$, and  integrating both of them on $Q_T$ we
obtain:
$$\int_{0}^{T}\int_{\Omega} [\partial_{t}y_{\varphi}^{*}p_{\varphi}+y_{\varphi}^{*}\partial_{t}p_{\varphi}]dx \ dt +
\int_{0}^{T}\int_{\Omega}[-d\Delta y_{\varphi}^{*} p_{\varphi}+d\Delta p_{\varphi} y_{\varphi}^{*}] dx \ dt +$$
$$+\int_{0}^{T}\int_{\Omega} [-a(x)y_{\varphi}^{*}p_{\varphi}+a(x)y_{\varphi}^{*}p_{\varphi}]dx \ dt =$$
$$= \int_{0}^{T}\int_{\Omega} [-H(\varphi (x))u_{\varphi}^{*}y_{\varphi}^{*}p_{\varphi} + H(\varphi (x))u_{\varphi}^{*}y_{\varphi}^{*} +
H(\varphi (x))u_{\varphi}^{*}y_{\varphi}^{*}p_{\varphi}]dx \ dt.$$
This means that
$$ -\int_{\Omega}y_{0}(x)p_{\varphi}(x,0) dx  = \int_{0}^{T}\int_{\Omega}H(\varphi (x))u_{\varphi}^{*}(x,t)y_{\varphi}^{*}(x,t) dx \ dt$$
and therefore
$$\int_{0}^{T}\int_{\omega}u_{\varphi}^{*}(x,t)y_{\varphi}^{*}(x,t) dx \ dt = - \int_{\Omega}y_{0}(x)p_{\varphi}(x,0) dx $$
Our problem of optimal harvesting becomes a problem of minimizing another functional with respect to the implicit function of $\omega$.
Therefore, we may rewrite the optimal problem as
$$ \underset{\varphi} {Minimize} \left\{\int_{\Omega }y_0(x)p_{\varphi }(x,0)dx
+\alpha\;\int_{\Omega} \delta (\varphi (x)) |\nabla \varphi(x)|
dx+\beta\;\int_{\Omega} H(\varphi (x)) dx\right\} .$$ By using
(\ref{f}) and (\ref{r}) we get that $p_{\varphi}$ is the solution
to
$$\left\{ \begin {array}{ll}
\partial _tp+d\Delta p=-a(x)p+LH(\varphi (x))(1+p)H(1+p), & (x,t)\in Q_T \\
\partial _{\nu }p(x,t)=0, & (x,t)\in \Sigma_T \\
p(x,T)=0, &x\in \Omega .
               \end{array}
  \right. $$
We shall adapt some shape optimization techniques to treat this
last harvesting problem (see also \cite{chanv}, \cite{getreuer}).
As  usual, we will approximate this problem by the following one,
where the Heaviside function $H$ is   substituted  by its
mollified version
$H_{\varepsilon}(t)=\frac{1}{2}\left(1+\frac{2}{\pi}\arctan\left(\frac{t}{\varepsilon}\right)\right),$
and its derivative by the mollified  function
$\delta_{\varepsilon}(t) =
\frac{\varepsilon}{\pi(\varepsilon^2+t^2)}.$

Therefore, for a small but fixed $\varepsilon >0$, the harvesting problem to be investigated is:
$$\underset{\varphi} {Minimize} \ J(\varphi ), $$
where $\varphi:\overline{\Omega}\longrightarrow\mathbb{R}$ is a smooth function,
$$J(\varphi )=\int_{\Omega }y_0(x)p_{\varphi }(x,0)dx
+\alpha\;\int_{\Omega} \delta _{\varepsilon }(\varphi (x)) |\nabla
\varphi(x)| dx+\beta\;\int_{\Omega} H_{\varepsilon }(\varphi (x))
dx,
$$ and $p_{\varphi}= p_{\varphi }(x,t)$ is the solution to
\begin{equation}
\left\{ \begin {array}{ll}
\partial _tp+d\Delta p=-a(x)p+LH_{\varepsilon }(\varphi (x))(1+p)H_{\varepsilon }(1+p), & (x,t)\in Q_T \\
\partial _{\nu }p(x,t)=0, & (x,t)\in \Sigma_T \\
p(x,T)=0, &x\in \Omega .
               \end{array}
  \right. \label{5.1}
\end{equation}
In the following we derive the directional derivative of $J$ (see \cite{ACM}).
\begin{teo}\label{t5.1} For any smooth functions $\varphi, \psi :\overline{\Omega}\longrightarrow\mathbb{R}$ we have that
$$dJ(\varphi )(\psi )=
\int_{\Omega} \delta_{\varepsilon}(\varphi(x))[-\alpha\;\ \textnormal{div}\left(\frac{\nabla \varphi(x)}{|\nabla \varphi(x)|}\right) +\beta\; $$
$$-L\int_0^T (1+p_{\varphi }(x,t))H_{\varepsilon }(1+p_{\varphi }(x,t))r_{\varphi }(x,t)dt]\psi(x) dx+\alpha\;\int_{\partial \Omega }{{\delta _{\varepsilon }(\varphi (x))}\over {|\nabla \varphi (x)|}}\partial _{\nu }\varphi (x)\psi (x)d\sigma ,$$
where $r_{\varphi}$ is the solution to
\begin{equation}\left\{ \begin{array}{ll}
\partial_t r- d\Delta r= a(x)r-LH_{\varepsilon }(\varphi(x))H_{\varepsilon }(1+p_{\varphi })r& ~ \\
 \ \ \ \ \ \ \ \ \ \ \ -LH_{\varepsilon }(\varphi(x))(1+p_{\varphi })\delta_{\varepsilon}(1+p_{\varphi })r, & (x,t)\in Q_T\\
\partial _{\nu } r(x,t)=0,  & (x,t)\in\Sigma_T\\
r(x,0)=y_0(x),  & x\in\Omega. \\
           \end{array}
    \right. \label{5.2}
\end{equation}
\end{teo}
\begin{proof}
It is possible to prove that
$$ {1\over {\theta }}[p_{\varphi+\theta \psi}-p_{\varphi}]\rightarrow q_{\psi} \quad \mbox{\rm in } C([0,T];L^{\infty }(\Omega )),$$
as $\theta \rightarrow 0$, where $q_{\psi}$ is the solution to the problem
\begin{equation}\left\{ \begin{array}{ll}
      \partial_t q+ d\Delta q= -a(x)q + L\delta_{\varepsilon}(\varphi(x))(1+p_{\varphi })H_{\varepsilon }(1+p_{\varphi }){\psi}(x)& ~ \\
 \ \ \ \ \ \ \ \ \ \ \      +LH_{\varepsilon }(\varphi(x))H_{\varepsilon }(1+p_{\varphi })q+ LH_{\varepsilon }(\varphi(x))(1+p_{\varphi })\delta_{\varepsilon}(1+p_{\varphi })q, & (x,t)\in Q_T\\
\partial _{\nu } q(x,t)=0,  & (x,t)\in\Sigma_T\\
q (x,T)=0,  & x\in\Omega. \\
           \end{array}
    \right. \label{5.3}\end{equation}
For any $\varphi $  we have that
$$\lim_{\theta \rightarrow 0}{1\over {\theta }}[J(\varphi +\theta \psi )-J(\varphi )]=\int_{\Omega }y_0(x)q_{\psi }(x,0)dx+ \alpha\;\int_{\Omega} \delta'_{\varepsilon}(\varphi(x))\psi (x)|\nabla\varphi(x)| dx $$
$$\begin{aligned}
& +\alpha\;\int_{\Omega }\delta _{\varepsilon }(\varphi (x)){{\nabla \varphi (x)\cdot \nabla \psi (x)}\over {|\nabla \varphi (x)|}}dx +  \beta\; \int_{\Omega}\delta _{\varepsilon }(\varphi (x))\psi (x) dx \\
&= \int_{\Omega}y_{0}(x)q_{\psi}(x,0) dx + \alpha\;\int_{\Omega} \text{div}(\delta_{\varepsilon}(\varphi(x)){{\nabla \varphi (x)}\over {|\nabla \varphi (x)|}}\psi (x))dx \\
&-\alpha\;\int_{\Omega }\delta _{\varepsilon }(\varphi (x))\psi (x) \text{div}\left({{\nabla \varphi (x)}\over {|\nabla \varphi (x)|}}\right)dx
+ \beta\; \int_{\Omega}\delta _{\varepsilon }(\varphi(x))\psi(x) dx.
\end{aligned}$$
Using Gauss-Ostrogradski's formula we get
\begin{equation}\begin{aligned}
dJ(\varphi )(\psi )=&\int_{\Omega}y_{0}(x)q_{\psi}(x,0) dx - \alpha\;\int_{\Omega} \delta_{\varepsilon}(\varphi(x))\text{div}\left(\frac{\nabla \varphi(x)}{|\nabla \varphi(x)|}\right)\psi(x) dx \\
&+ \beta\; \int_{\Omega}\delta_{\varepsilon}(\varphi(x))\psi(x) dx+\alpha\;\int _{\partial \Omega }{{\delta _{\varepsilon }(\varphi (x))}\over {|\nabla \varphi (x)|}}\partial _{\nu }\varphi (x)\psi (x)d\sigma .
\end{aligned}
\label{5.4}\end{equation}
By multiplying the first equation in \eqref{5.3} by $r_{\varphi}$ and integrating
over $Q_T$, using \eqref{5.2} we get that
\begin{equation}\int_{\Omega} y_0(x)q_{\psi }(x,0)dx = -\int_0^T \int_{\Omega} Lr_{\varphi}(x,t)\delta_{\varepsilon}(\varphi(x))(1+p_{\varphi }(x,t))H_{\varepsilon }(1+p_{\varphi }(x,t))\psi(x) dx\ dt. \label{5.5}\end{equation}
Now, from \eqref{5.4} and \eqref{5.5}, we get the conclusion of Theorem \ref{t5.1}.
\end{proof}
Let  us   remark that the gradient descent with respect to
$\varphi $ is
\begin{equation}\left\{ \begin{array}{ll}
      \partial_{\theta } \varphi(x,\theta ) = \delta_{\varepsilon}(\varphi (x,\theta ))[\alpha\; \text{\rm div}
      \left( \frac{\nabla \varphi(x,\theta )}{|\nabla \varphi(x, \theta )|}\right)-\beta&\\
 \ \ \ \ \ \ \  \ \ \ \ \ \ \     +L\int_0^T (1+p_{\varphi }(x,t))H_{\varepsilon }(1+p_{\varphi }(x,t))r_{\varphi }(x,t)dt ], & x \in \Omega , \ \theta >0 \\
\frac{\delta_{\varepsilon}(\varphi(x,\theta ))}{|\nabla \varphi(x,\theta )|}\partial_{\nu} \varphi (x,\theta )=0, &  x\in \partial\Omega , \ \theta >0 .
           \end{array} \right. \label{5.6}\end{equation}
($\theta$ is an artificial time).

\subsection{ Numerical implementation}

From Theorem \ref{t5.1} we derive the following conceptual iterative algorithm, a semi-implicit gradient descent method, to improve at each step the region where the harvesting effort acts in order to obtain a smaller value for $J $.
\vspace{3mm}

\tt
\noindent\textbf{STEP 0:} set $n := 0$, $J^{(0)} := 10^6$ and $\theta _0>0$ a small constant

\hspace{1cm}initialize $\varphi^{(0)}=\varphi^{(0)}(x,0)$

\vspace{1.5mm}

\noindent \textbf{STEP 1:} compute  $p^{(n+1)}$ the solution of \eqref{5.1} corresponding to $\varphi^{(n)}(\cdot ,0)$

\vspace{1.5mm}

\hspace{1cm}compute $\displaystyle J^{(n+1)} =
\int_{\Omega }y_0(x)p^{(n+1)}(x,0)dx$

\hspace{2.6cm} $+\alpha\;\int_{\Omega} \delta _{\varepsilon }(\varphi ^{(n)}(x,0)) |\nabla \varphi ^{(n)}(x,0)| dx\displaystyle +\beta\;\int_{\Omega} H_{\varepsilon }(\varphi ^{(n)}(x,0)) dx$.

\vspace{1.5mm}

\noindent \textbf{Step 2:} if $\left|J^{(n+1)} - J^{(n)}\right| < \varepsilon_1$ or $J^{(n+1)} \geq J^{(n)}$ then \textbf{STOP}

\hspace{1cm}else go to \textbf{Step 3}.

\noindent \textbf{Step 3:} compute $r^{(n+1)}$ the solution of problem \eqref{5.2} corresponding to

\hspace{1cm}$\varphi^{(n)}(\cdot, 0)$ and $p^{(n+1)}$.

\vspace{1.5mm}

\noindent\textbf{Step 4:} compute $\varphi^{(n+1)}$ using \eqref{5.6} and the initial condition

\hspace{1cm}$\varphi ^{(n+1)}(x,0)=\varphi ^{(n)}(x,\theta _0)$ and a semi-implicit timestep scheme

\vspace{1.5mm}

\noindent\textbf{Step 5:} if $\|\varphi^{(n+1)}-\varphi^{(n)}\|_{L^2}<\varepsilon_2$ then \textbf{STOP}

\hspace{1cm}else $n := n + 1$

\hspace{1.9cm}go to \textbf{Step 1}

\rm
\vspace{3mm}

\noindent $\varepsilon_1 > 0$ in Step 2 and $\varepsilon_2 > 0$ in Step 5 are prescribed convergence parameters.
\vspace{3mm}

For the implementation we consider $\Omega=(0, 1) \times (0, 1)$ such that the sides are parallel with $Ox_1$ and $Ox_2$ axes. We introduce equidistant discretization nodes for both axes
corresponding to $\Omega$. Thus, the domain $\Omega$ is
approximated by a grid of $(N+1)\times (N+1)$ equidistant nodes,
namely $$\{(x_1^i, x_2^j):\ x_1^i = (i-1)\Delta x_1, x_2^j =
(j-1)\Delta x_2,  i,j=\overline{1,N+1}, \Delta x_1=\Delta x_2 =
1/N\}.$$
The interval $[0,T]$ is also discretized by $M+1$ equidistant nodes,
$t^k = (k-1)\Delta t, k=1, 2,...M, M+1, \Delta t=T/M$. We take $M$ and $N$ to be even. We denote by $\varphi_{ij} = \varphi(x_1^i,x_2^j), i,j = \overline{1, N+1}$.

In order to approximate the solution of the parabolic system from Step 1 we use a finite difference method, an implicit one, descending with respect to time levels.

We denote by $h = \Delta x_1=\Delta x_2$, $p_{ij}^{k} = p(x_1^i, x_2^j, t^k)$, $a_{ij} = a(x_1^i, x_2^j)$, $G_{ij}^{k} =\Delta t L H_{\varepsilon }(\varphi_{ij})(1+ p_{ij}^{k+1})H_{\varepsilon }(1+ p_{ij}^{k+1})$, $k=\overline{1,M+1}$. The numerical scheme is
$$
\left\{ \begin{array}{ll}
\frac{p_{ij}^{k+1}-p_{ij}^{k}}{\Delta t}+d\frac{p_{i-1,j}^{k}-2p_{ij}^{k}+p_{i+1,j}^{k}}{h^2} +d\frac{p_{i,j-1}^{k}-2p_{ij}^{k}+p_{i,j+1}^{k}}{h^2}\\
~~~~~~~~~+a_{ij}p_{ij}^{k}-G_{ij}^{k+1}=0, &i,j = \overline{2,N},\\ &k = M, \dots,1,\\
p_{i,1}^{k} = p_{i,2}^{k}, p_{i,N+1}^{k} = p_{i,N}^{k}, p_{1,j}^{k} = p_{2,j}^{k}, p_{N+1,j}^{k} = p_{N,j}^{k}, &i,j = \overline{1,N+1}, \\&k = M, \dots,1,\\
p_{ij}^{M+1} = 0, &i, j= \overline{1, N+1}.
\end{array}
\right.
$$
We  take the diffusion coefficient $d = 1$ for the implementation,
and denote by $\lambda = \Delta t/h^2.$  For the interior nodes we
get
$$(1 + 4\lambda - \Delta t \ a_{ij})p_{ij}^{k}-\lambda p_{i-1,j}^{k}-\lambda p_{i+1,j}^{k} -\lambda p_{i,j-1}^{k}-\lambda p_{i,j+1}^{k}
=p_{ij}^{k+1}-\Delta t G_{ij}^{k+1}$$
for $i,j = \overline{2,N}, k = M, \dots,1$.

by using the Neumann conditions on the boundary, the numerical scheme becomes
\begin{equation}\label{s1}
\left\{
\begin{array}{ll}
(1 + 2\lambda - \Delta t\ a_{2,2})p_{2,2}^{k}-\lambda p_{2,3}^{k}-\lambda p_{3,2}^{k}
=p_{2,2}^{k+1}-\Delta t G_{2,2}^{k+1},~~~~~~~ \ i=2, \ j = 2\\
(1 + 3\lambda - \Delta t \ a_{2,j})p_{2,j}^{k}-\lambda p_{3,j}^{k}-\lambda p_{2,j-1}^{k}-\lambda p_{2,j+1}^{k}
=p_{2,j}^{k+1}-\Delta t G_{2,j}^{k+1},&\\
~~~~~~~~~~~~~~~~~~~~~~~~~~~~~~~~~~~~~~~~~~~~~~~~~~~~~~~~~~~~~~~~~~~~~~~~~~~~~i=2, \ 2<j<N\\
(1 + 2\lambda - \Delta t\ a_{2,N})p_{2,N}^{k}-\lambda p_{3,N}^{k}-\lambda p_{2,N-1}^{k}
=p_{2,N}^{k+1}-\Delta t G_{2,N}^{k+1}, \ i = 2, \ j = N\\
(1 + 3\lambda - \Delta t \ a_{i,2})p_{i,2}^{k}-\lambda p_{i,3}^{k}-\lambda p_{i-1,2}^{k}-\lambda p_{i+1,2}^{k}
=p_{i,2}^{k+1}-\Delta t G_{i,2}^{k+1},\\
~~~~~~~~~~~~~~~~~~~~~~~~~~~~~~~~~~~~~~~~~~~~~~~~~~~~~~~~~~~~~~~~~~~~~~~~~~~~2<i<N, \ j=2\\
(1 + 4\lambda - \Delta t \ a_{ij})p_{ij}^{k}-\lambda p_{i-1,j}^{k}-\lambda p_{i+1,j}^{k} -\lambda p_{i,j-1}^{k}-\lambda p_{i,j+1}^{k}
=p_{ij}^{k+1}-\Delta t G_{ij}^{k+1},\\
~~~~~~~~~~~~~~~~~~~~~~~~~~~~~~~~~~~~~~~~~~~~~~~~~~~~~~~~~~~~~~~~~~~~~~2<i<N, \ 2<j<N\\
(1 + 3\lambda - \Delta t \ a_{i,N})p_{i,N}^{k}-\lambda p_{i,N-1}^{k}-\lambda p_{i-1,N}^{k}-\lambda p_{i+1,N}^{k}
=p_{i,N}^{k+1}-\Delta t G_{i,N}^{k+1},\\
~~~~~~~~~~~~~~~~~~~~~~~~~~~~~~~~~~~~~~~~~~~~~~~~~~~~~~~~~~~~~~~~~~~~~~~~~~~2<i<N, \  j=N\\
(1 + 2\lambda - \Delta t\ a_{N,2})p_{N,2}^{k}-\lambda p_{N,3}^{k}-\lambda p_{N-1,2}^{k}
=p_{N,2}^{k+1}-\Delta t G_{N,2}^{k+1}, \ i=N, \ j=2\\
(1 + 3\lambda - \Delta t \ a_{N,j})p_{N,j}^{k}-\lambda p_{N-1,j}^{k}-\lambda p_{N,j-1}^{k}-\lambda p_{N,j+1}^{k}
=p_{N,j}^{k+1}-\Delta t G_{N,j}^{k+1},\\
~~~~~~~~~~~~~~~~~~~~~~~~~~~~~~~~~~~~~~~~~~~~~~~~~~~~~~~~~~~~~~~~~~~~~~~~~~~i=N, \ 2<j<N\\
(1 + 2\lambda - \Delta t\ a_{N,N})p_{N,N}^{k}-\lambda p_{N-1,N}^{k}-\lambda p_{N,N-1}^{(k)}
=p_{N,N}^{k+1}-\Delta t G_{N,N}^{k+1},\\
~~~~~~~~~~~~~~~~~~~~~~~~~~~~~~~~~~~~~~~~~~~~~~~~~~~~~~~~~~~~~~~~~~~~~~~~~~~~~~~~i=N, \ j = N.
\end{array}
\right.
\end{equation}
We denote by $$x^{k} = (p_{2,2}^{k}, p_{2,3}^{k}, \dots,p_{2,N}^{k},p_{3,2}^{k},p_{3,3}^{k}, \dots, p_{3,N}^{k},\dots,p_{N,2}^{k}, p_{N,3}^{k},\dots,p_{N,N}^{k})^T$$
the vector formed by the values of $p$ at time level $k$ for the interior nodes. This is a vector of dimension $(N-1)^2$. We also use the following notations $P = p_{ij}^{k+1}$, $G =\Delta t G_{ij}^{k+1}$, $E_1 = \Delta t(- a_{ij})$, $E_2 = 1 + 2\lambda + E_1$, $E_3 = 1 + 3\lambda + E_1$, $E_4 = 1 + 4\lambda + E_1$. This quantities must be evaluated at each time step $k = M, M-1, \dots,1$ and for all $i, j = \overline{2,N}$. The algebraic linear system to solve at each time step $k = M, M-1, \dots,1$ is of the form $Ax^{k} = B$, with the system matrix $A$ of dimension $(N-1)^2\times(N-1)^2$ and the vector of constant terms $B$ of dimension $(N-1)^2$. Based on \eqref{s1} and using also the final condition, for each time level $k = M, M-1, \dots,1$ we generate the matrix $A$ and the vector $B$ with the following algorithm: we denote by $q$ the row index of matrix $A$; at the beginning of each time iteration we make the initializations: $q = 0$, $A = 0_{(N-1)^2\times(N-1)^2}$, and $B = 0_{(N-1)^2\times 1}$. Then, for $i$ from $2$ to $N$ and for $j$ from $2$ to $N$, after the evaluation of $E_1, E_2, E_3, E_4, G, P$, we start the construction of $A$ and $B$. The index $q$ is incremented for each $i$ and $j$. Therefore,
\begin{itemize}
\setlength{\itemsep}{0cm}
\setlength{\parskip}{0cm}
\item\texttt{if i = 2 and j = 2 then q = q + 1; A(q,1) = $E_2$; A(q,2) = -\lam; \\A(q,N) = -\lam; B(q) = P - G;}\\
\item\texttt{if i = 2 and 2 < j < N then q = q + 1; A(q,j-2) = -\lam;\\ A(q,j-1) = $E_3$; A(q,N+j-2) = -\lam; A(q,j) = -\lam; B(q) = P - G;}\\
\item\texttt{if i = 2 and j = N then q = q + 1; A(q,N-2) = -\lam; \\A(q,N-1) = $E_2$; A(q,2*N-2) = -\lam; B(q) = P - G;}\\
\item\texttt{if 2 < i < N and j = 2 then q = q + 1;\\ A(q,(i-3)*(N-1)+1) = -\lam; A(q,(i-2)*(N-1)+1) = $E_3$;\\ A(q,(i-1)*(N-1)+1) = -\lam; A(q,(i-2)*(N-1)+2)= -\lam; B(q) = P - G;}\\
\item\texttt{if 2 < i < N and 2 < j< N then q = q + 1;\\ A(q,(i-3)*(N-1)+j-1) = -\lam; A(q,(i-1)*(N-1)+j-1) = -\lam\\; A(q,(i-2)*(N-1)+j-1) = $E_4$;\ A(q,(i-2)*(N-1)+j-2) = -\lam;\\ A(q,(i-2)*(N-1)+j) = -\lam; B(q) = P - G;}\\
\item\texttt{if 2 < i < N and j = N then q = q + 1; A(q,(i-2)*(N-1)) = -\lam;\\ A(q,(i-2)*(N-1)+N-2) = -\lam; A(q,(i-1)*(N-1)) = $E_3$;\\ A(q,i*(N-1)) = -\lam; B(q) =P - G;}\\
\item\texttt{if i = N and j = 2 then q = q + 1; A(q,(N-3)*(N-1)+1) = -\lam;\\ A(q,(N-2)*(N-1)+1) = $E_2$; A(q,(N-2)*(N-1)+2) = -\lam; B(q) = P - G;}\\
\item\texttt{if i = N and 2 < j < N then q = q + 1;\\ A(q,(N-3)*(N-1)+j-1) = -\lam; A(q,(N-2)*(N-1)+j-2) = -\lam;\\ A(q,(N-2)*(N-1)+j-1) = $E_3$; A(q,(N-2)*(N-1)+j) = -\lam; B(q) = P - G;}\\
\item\texttt{if i = N and j = N then q = q + 1; A(q,N*(N-2)) = -\lam;\\ A(q,(N-2)*(N-1)) = -\lam; A(q,(N-1)*(N-1)) = $E_2$; B(q) = P - G;}
\end{itemize}
Then, the resulting algebraic linear system is solved by Gaussian elimination.  The solution obtained is a vector $D$ of dimension $(N-1)^2$.
Therefore, we get the corresponding solution $p_{ij}$ at time step $k$ by the process:\\
\texttt{q = 0;}\\
\texttt{for i = 2 to N}

\texttt{for j = 2 to N}

\hspace{0.3cm} \texttt{q = q + 1; $p_{ij}^{k}$ = D(q);}

By using the boundary condition, the solution is completed for
$i=1$, $i = N+1$ and $j = \overline{1,N+1}$ and for $j=1$, $j =
N+1$ and $i = \overline{1,N+1}$. Now we have the complete solution
$p_{ij}^{k}$ and we can proceed with the time step $k-1$.

The integrals from Step 1 are numerical computed using Simpson's method corresponding to the discrete grid. For each iteration $n=1, 2, 3, \dots$, we have to evaluate the first integral
$$F^{(n)} = \int\limits_{\Omega}f^{(n)}(x)dx,$$
where
$$f^{(n)}(x) = y_0(x)p^{(n)}(x,0), x \in \Omega.$$
In order to approximate this integral we first calculate, for all $i = \overline{1,N+1}$,
$$r(i) = \frac{h}{3}\left[f^{(n)}(x_1^i,x_2^1)+f^{(n)}(x_1^i,x_2^{N+1})+4\sum_{\substack{j = 2\\j = j+2}}^Nf^{(n)}(x_1^i,x_2^j)+ 2\sum_{\substack{j = 3\\j = j+2}}^{N-1}f^{(n)}(x_1^i,x_2^j)\right],$$
and then
$$F^{(n)}\approx\frac{h}{3}\left[r(1)+r(N+1)+4\sum_{\substack{i = 2\\i = i+2}}^Nr(i)+ 2\sum_{\substack{i = 3\\i = i+2}}^{N-1}r(i)\right].$$ To numerical evaluate of the second integral we must approximate $|\nabla \varphi(x_1^i, x_2^j)|$. In order to do this, we use central difference both in $x_1$ and in $x_2$ direction.
$$|\nabla \varphi(x_1^i, x_2^j)| = \sqrt{(\partial_{x_1}\varphi(x_1^i, x_2^j))^2+(\partial_{x_2}\varphi(x_1^i, x_2^j))^2}$$
$$=\sqrt{\frac{(\varphi_{i+1,j}-\varphi_{i-1,j})^2+(\varphi_{i,j+1}-\varphi_{i,j-1})^2}{4h^2}}, \ i,j = \overline{2, N}$$
$$|\nabla \varphi(x_1^1, x_2^j)|=\sqrt{\frac{(\varphi_{2,j}-\varphi_{1,j})^2+(\varphi_{1,j+1}-\varphi_{1,j})^2}{h^2}}, \ j = \overline{2, N}$$
$$|\nabla \varphi(x_1^{N+1}, x_2^j)|=\sqrt{\frac{(\varphi_{N+1,j}-\varphi_{N,j})^2+(\varphi_{N+1,j+1}-\varphi_{N+1,j})^2}{h^2}}, \ j = \overline{2, N}$$
$$|\nabla \varphi(x_1^i, x_2^1)|=\sqrt{\frac{(\varphi_{i+1,1}-\varphi_{i,1})^2+(\varphi_{i,2}-\varphi_{i,1})^2}{h^2}}, \ i = \overline{2, N}$$
$$|\nabla \varphi(x_1^i, x_2^{N+1})|=\sqrt{\frac{(\varphi_{i+1,N+1}-\varphi_{i,N+1})^2+(\varphi_{i,N+1}-\varphi_{i,N})^2}{h^2}}, \ i = \overline{2, N}$$
$$|\nabla \varphi(x_1^1, x_2^{1})| = |\nabla \varphi(x_1^2, x_2^{1})|, |\nabla \varphi(x_1^1, x_2^{{N+1}})| = |\nabla \varphi(x_1^1, x_2^{N})|,$$
$$|\nabla \varphi(x_1^{N+1}, x_2^{1})| = |\nabla \varphi(x_1^{N+1}, x_2^{2})|, |\nabla \varphi(x_1^{N+1}, x_2^{N+1})| = |\nabla \varphi(x_1^{N+1}, x_2^{N})|.$$
The parabolic system from Step 3 is approximated also using a
finite difference method, but now ascending with respect to time
levels. For each iteration $n = 1, 2, \dots$ and for each time
level $k = 1, 2, 3,\dots,M$, the matrix of the resulting algebraic
system is the same as matrix A previously determinated, with $B(q)
= p_{ij}^{k}, q = \overline{1, (N-1)^2}$, and $E_1 = \Delta t
(-a_{ij}+G_{ij})$, $G_{ij} = L H_{\varepsilon
}(\varphi_{ij})H_{\varepsilon }(1+ p_{ij}^{k+1})+LH_{\varepsilon
}(\varphi_{ij})(1+ p_{ij}^{k+1})\delta_{\varepsilon }(1+
p_{ij}^{k+1})$, which are evaluated for each $i,j = \overline{2,
N}$. The resulting algebraic linear system is solved by Gaussian
elimination. By using the boundary conditions we complete the
solution of the parabolic system for each time level.

\eject

\noindent{\bf Numerical examples}

\vspace{2mm}

\noindent We consider a normal initial population density $y_0(x_1,x_2) = \frac{1}{2\pi} e^{-\frac{x_1^2+x_2^2}{2}}$, where $(x_1, x_2)\in \Omega$. Let the diffusion coefficient be $d=1$, the final time $T=1$, $L=1$, and the regularization parameter $\varepsilon = 1$. We take the space discretization step and the time discretization step to be equal $\Delta x_1=\Delta x_2=\Delta t=0.05$. For the convergence tests we consider $\varepsilon_1 = \varepsilon_2 =0.001$.

\noindent In the following figure, the white area represents the subregion $\omega$ that provides a small value for $J$.

\vspace{1.5mm}

\noindent {\bf Test 1.} We take the natural growth rate of the population to be a constant, e.g. $a (x_1,x_2) = 3$, $(x_1,x_2) \in \Omega$. The initialization of $\varphi$ is made by $\varphi^{(0)}(x_1,x_2)= 0.25 - \sqrt{(x_1-0.5)^2+(x_2-0.5)^2}$, $(x_1, x_2)\in \Omega$.  We penalize the length of $\partial \omega$ by $\alpha=0.4$ and the area of $\omega$ by $\beta = 0.6$. The corresponding results are shown in Figure 1.

\begin{figure}[!h]
\centering
\subfigure[Initial $\omega$]{\label{fig:1_2}
\includegraphics[width=0.4\textwidth]{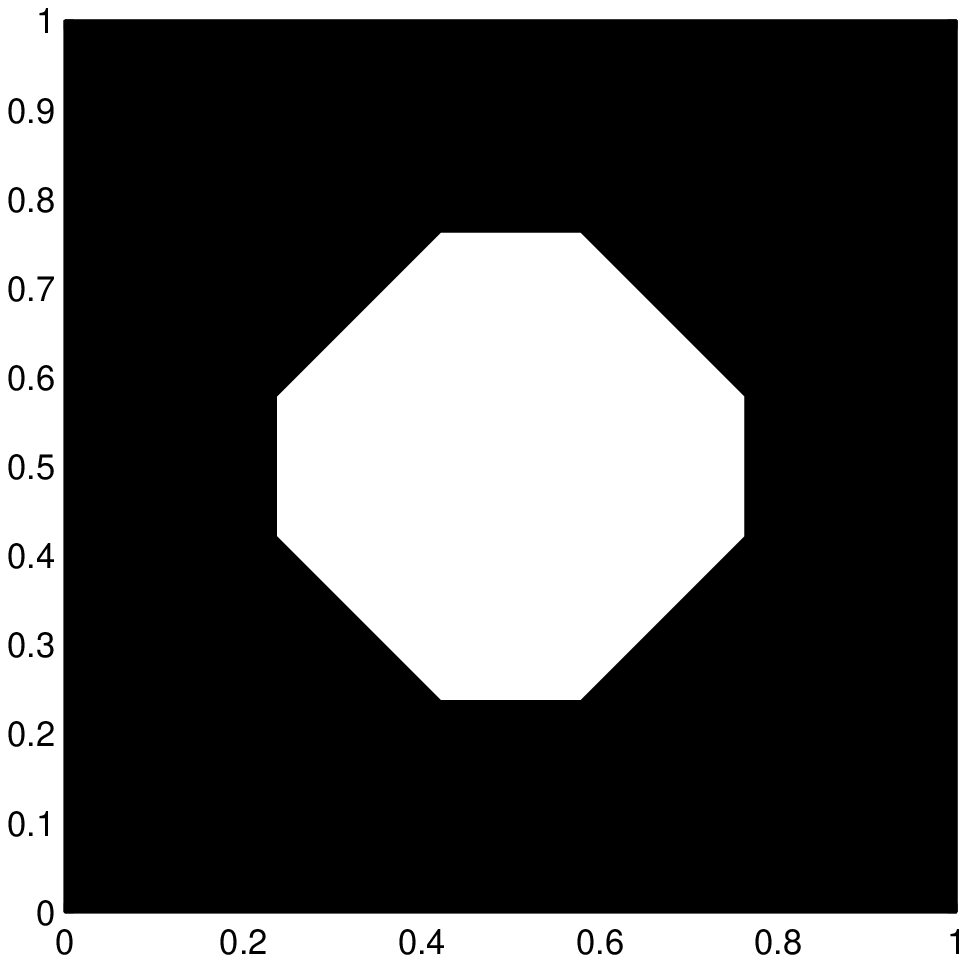}}
\subfigure[Final $\omega$]{\label{fig:2_2}
\includegraphics[width=0.4\textwidth]{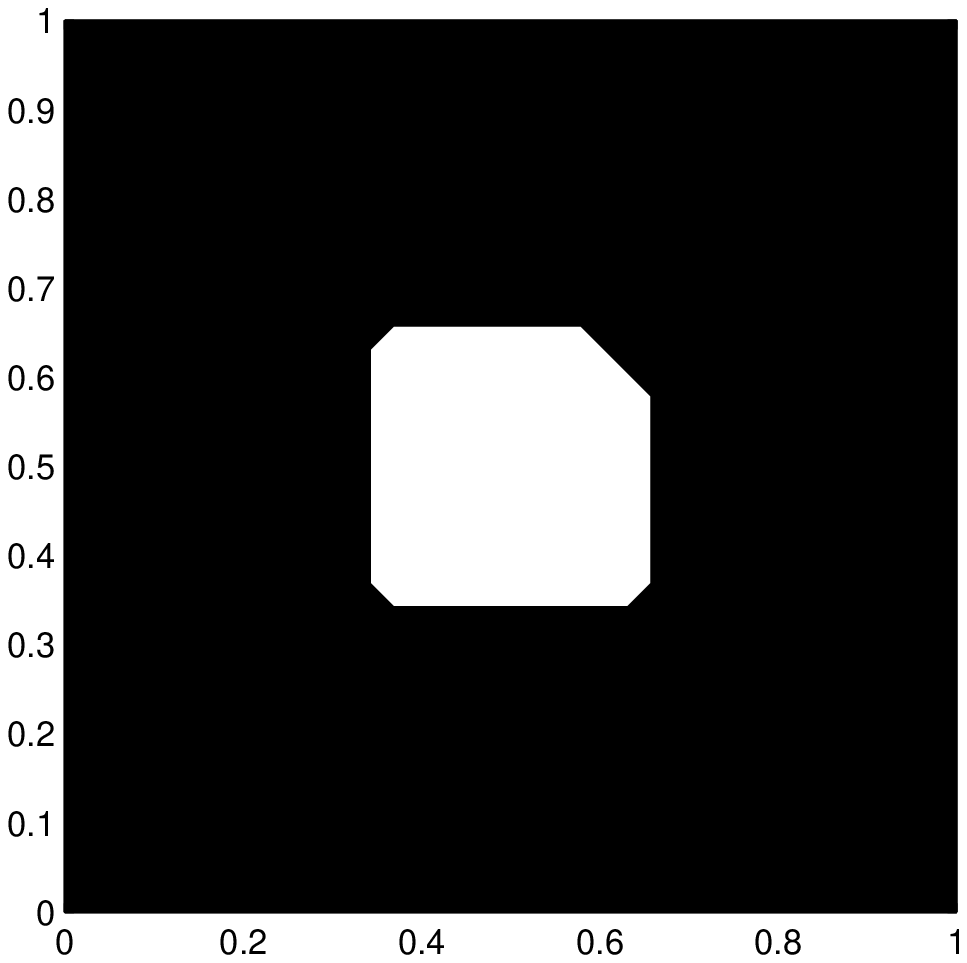}}
\caption{The representation of initial and final iterations of $\omega$ for $\alpha = 0.4$ and $\beta =0.6$.}
\label{fig:C_2}
\end{figure}

\vspace{1.5mm}

\noindent {\bf Test 2.} We use the same input data from Test 1 and the initialization of $\varphi$ with $\varphi^{(0)}(x_1,x_2)=sin(3\pi x_1)sin(3\pi x_2)$, $(x_1, x_2)\in \Omega$, a function that produce a initial checkerboard shape.  We penalize the length of $\partial \omega$ by $\alpha=0.5$ and the area of $\omega$ by $\beta = 0.5$. The results are shown in Figure 2.

\begin{figure}[!h]
\centering
\subfigure[Initial $\omega$]{\label{fig:1_3}
\includegraphics[width=0.4\textwidth]{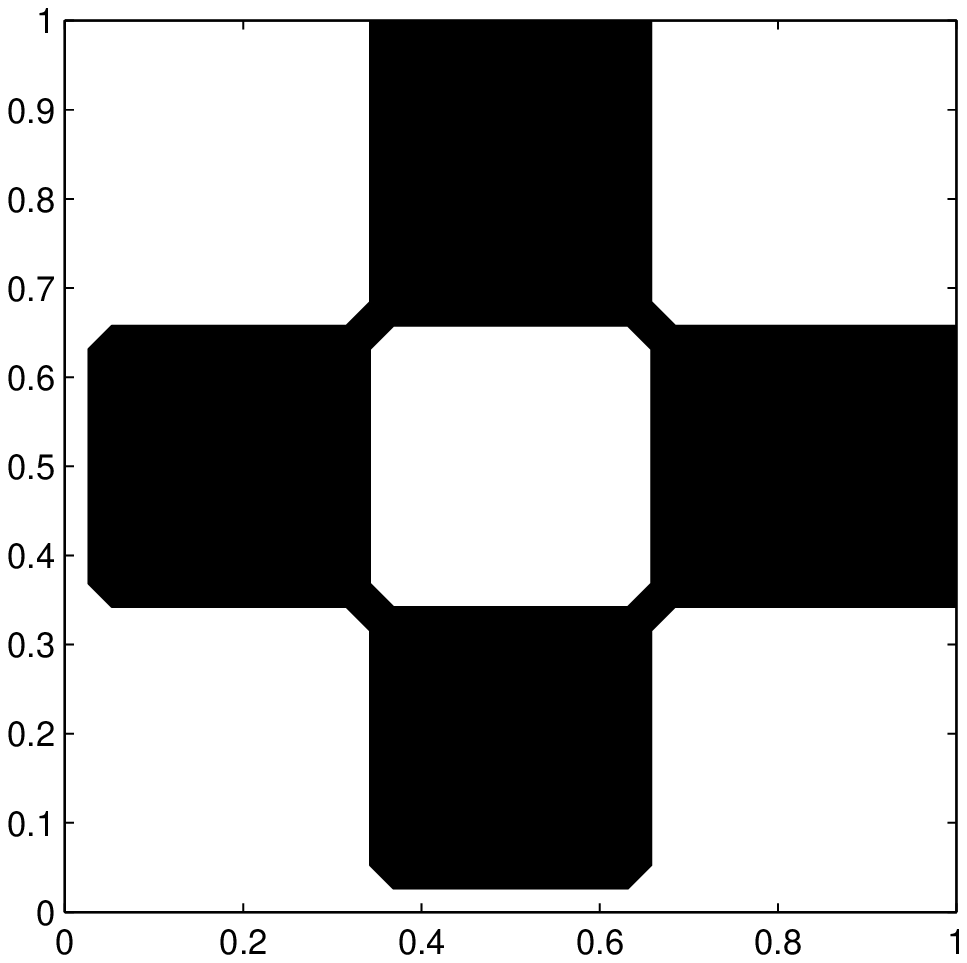}}
\subfigure[Final $\omega$]{\label{fig:2_3}
\includegraphics[width=0.4\textwidth]{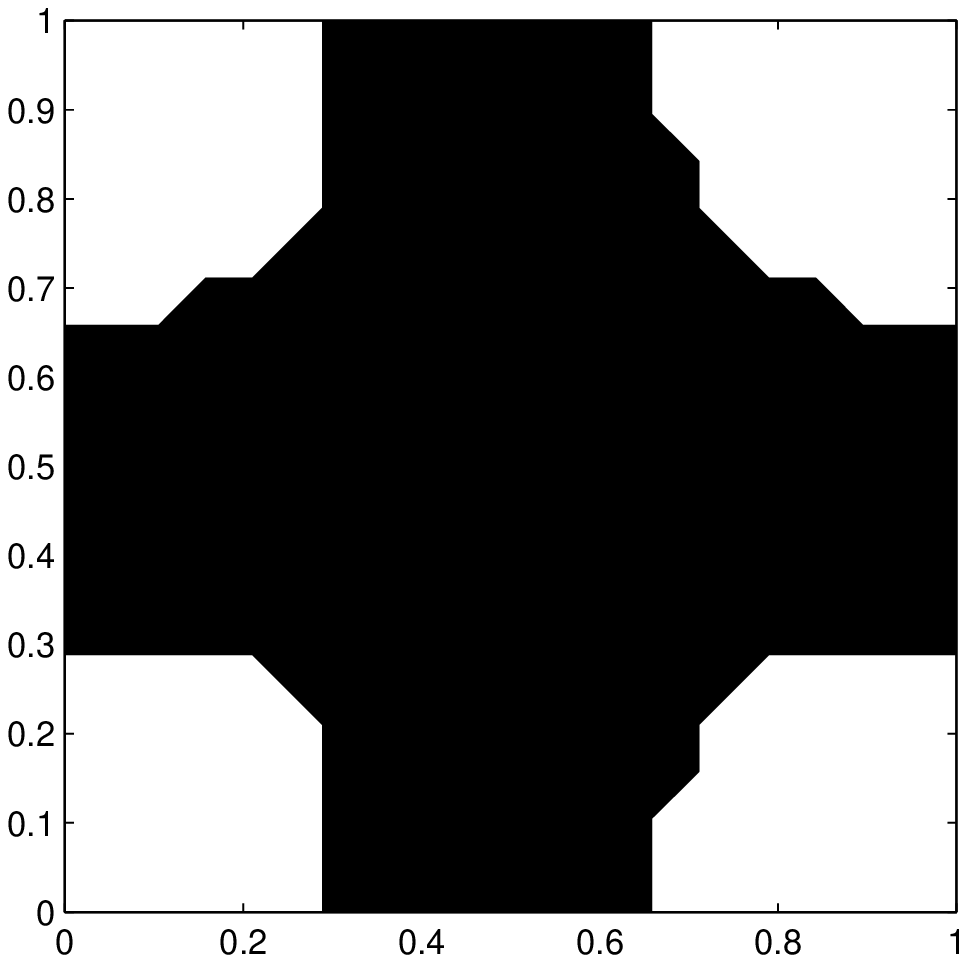}}
\caption{The representation of initial and final iterations of $\omega$ for $\alpha = 0.5$ and $\beta =0.5$.}
\label{fig:C_3}
\end{figure}

\section{Eradicating an age-structured pest population with diffusion}
Consider here an age-structured population dynamics with diffusion and logistic term:
\begin{equation}\left\{ \begin {array}{ll}
\partial _ty(x,a,t)+\partial _ay(x,a,t)+\mu (a)y(x,a,t) -d\Delta y(x,a,t)=\\
 \ \ \ \ \ \ -\mathcal{M}\left(\int_0^Ay(x,a,t)da\right)y(x,a,t)\\
 \ \ \ \ \ \ -\chi_{\omega }(x)u(x,t)y(x,a,t),~~~~~~~~~~~~x\in\Omega, \ a\in (0,A), \ t\in(0,+\infty) \\
\partial _{\nu }y(x,a,t)=0,~~~~~~~~~~~~~~~~~~~~~~~~~~~x\in\partial\Omega, \ a\in (0,A), \ t\in(0,+\infty)\\
y(x,0,t)=\int_0^A\beta (a)y(x,a,t)da, ~~~~~~~~x\in\Omega, \ t\in(0,+\infty)\\
y(x,a,0)=y_0(x,a), ~~~~~~~~~~~~~~~~~~~~~~x\in\Omega, \ a\in (0,A).
               \end{array}
  \right. \label{4.1}\end{equation}
Here $A\in (0,+\infty )$ is the maximal age for the population
species and $y(x,a,t)$ is the population density at position $x$,
age $a$ and time $t$; $d\in (0,+\infty )$ is the diffusion
coefficient, $\mu (a)$ is the mortality rate and $\beta (a)$ is
the fertility rate for individuals of age $a$; $y_0(x,a)$ is the initial density of population at position $x$ and age $a$.
$u(x,t)$ is a harvesting effort (the control) and is localized in the subregion $\omega$; $u$ does not depend on age.

Assume that $\Omega$, $\omega$ satisfy the same assumptions as in
the introduction, and that  the following hypotheses are satisfied
as well
\begin{itemize}
\item[{\bf (H1')}] $\beta \in C([0,A])$, $\beta (a)\geq 0, \forall a\in [0,A]$;
\item[{\bf (H2')}] $\mu \in C([0,A))$, $\mu (a)\geq 0, \forall a\in [0,A]$, $\int_0^A \mu(a)da = +\infty$;
\item[{\bf (H3')}] $y_0\in L^{\infty }(\Omega \times (0,A))$, $y_0(x,a)\geq 0$ a.e. in $\Omega \times (0,A)$;
\item[{\bf (H4')}] $\mathcal{M}:[0,+\infty)\longrightarrow[0,+\infty)$ is continuously differentiable, $\mathcal{M}'(r)>0, \forall r>0,  \mathcal{M}(0)=0, \lim\limits_{r\to+\infty}\mathcal{M}(r) = +\infty$.
\end{itemize}
For any $u \in L_{loc}^{\infty}(\overline{\omega}\times[0,+\infty))$, such that $L\geq u(x,t) \geq 0$ a.e., there exists a unique solution $y^u$ to \eqref{4.1} (here $L\in (0,+\infty)$ is a constant; $L$ is the maximal affordable effort). This solution is nonnegative (see \cite{ACM}).

Our goal is to eradicate this population which is considered to be a pest population.
\begin{defi}
We say that the population is eradicable (zero-stabilizable) if for any $y_0$ satisfying the hypothesis (H3') there exists $u \in L_{loc}^{\infty}(\overline{\omega}\times[0,+\infty))$, satisfying $L\geq u(x,t) \geq 0$ a.e., such that
$$\lim\limits_{t\to+\infty}y^u(\cdot,\cdot,t) = 0 \text{ in } L^{\infty}(\Omega\times(0,A)).$$
(and $y^u(x,a,t) \geq 0$ a.e. in $\Omega \times (0,A) \times (0, +\infty)$).
\end{defi}
Note that this is a problem of zero-stabilization with control and state constraints.

Denote by $r^*\in \mathbb{R}$ the solution to the equation
$$\int_0^A\beta(a)e^{-\int_0^a\mu(\tau)d\tau-ra}da = 1$$
and $\lambda_1^\omega$ the principal eigenvalue for
$$\left\{ \begin {array}{ll}
-d\Delta \phi=-\chi_{\omega}L\phi,&~x\in\Omega\\
\partial _{\nu }\phi=0,&~x\in\partial\Omega.\\
      \end{array}
  \right. $$
\begin{teo}
(i) If the population is eradicable then
$$\lambda_1^\omega \geq r^*.$$
(ii) If $\lambda_1^\omega > r^*$ then the population is eradicable and the harvesting effort $u\equiv L$ diminishes exponentially the population.
\end{teo}
\begin{proof}
(i) Assume that the population is eradicable and let
$$y_0(x,a) = h_0(a)g_0(x),$$
with $h_0\in C([0,A]), h_0(a)>0, \forall a\in[0,A]$, to be specified later, $g_0\in L^{\infty}(\Omega), g_0(x)>0$ a.e in $\Omega$.

Let $u \in L_{loc}^{\infty}(\overline{\omega}\times[0,+\infty))$, $L\geq u(x,t) \geq 0$ a.e., such that
$$\lim\limits_{t\to+\infty}y^u(\cdot,\cdot,t) = 0 \text{ in } L^{\infty}(\Omega\times(0,A)).$$

The unique solution $y^u$ to \eqref{4.1} may be written as
$$y^u(x,a,t) = h(a,t)g(x,t),$$
where $h$ is the solution to
\begin{equation}\left\{ \begin {array}{ll}
\partial _th(a,t)+\partial _ah(a,t)+\mu (a)h(a,t) = -r^*h(a,t), &~ a\in (0,A), \ t\in(0,+\infty) \\
h(0,t)=\int_0^A\beta (a)h(a,t)da,&~ t\in(0,+\infty)\\
h(a,0)=h_0(a), &~ a\in (0,A),
               \end{array}
  \right. \label{4.3}\end{equation}
and $g$ is the solution to
\begin{equation}\left\{ \begin {array}{ll}
\partial _tg(x,t)-d\Delta g(x,t)=r^*g(x,t)\\
\ \ \ \ \ \ \ -\mathcal{M}\left(\int_0^Ah(a,t)g(x,t)da\right)g(x,t)\\
 \ \ \ \ \ \ -\chi_{\omega }(x)u(x,t)g(x,t), &~ x\in\Omega, \ t\in(0,+\infty) \\
\partial _{\nu }g(x,t)=0, &~x\in\partial\Omega, \ t\in(0,+\infty)\\
g(x,0)=g_0(x), &~x\in\Omega.
               \end{array}
  \right. \label{4.4}\end{equation}
It is known that the set of solutions for $\eqref{4.3}_{1-2}$ is a
real vector space of dimension $1$ and there exists a time
independent solution $\tilde{h}$ satisfying $\tilde{h}(a)>0,$ for
all $a\in[0,A)$ (\cite{anita}).

If we consider $h_0 = \tilde{h}$, then
$$y^u(x,a,t) = \tilde{h}(a)g(x,t),\ x\in\Omega, \ a\in (0,A), \ t\in(0,+\infty),$$
where $g$ is the solution to
\begin{equation}\left\{ \begin {array}{ll}
\partial _tg(x,t)-d\Delta g(x,t)=r^*g(x,t)\\
\ \ \ \ \ \ \ -\mathcal{M}\left(Hg(x,t)\right)g(x,t)\\
 \ \ \ \ \ \ -\chi_{\omega }(x)u(x,t)g(x,t), &~ x\in\Omega, \ t\in(0,+\infty) \\
\partial _{\nu }g(x,t)=0, &~x\in\partial\Omega,\  t\in(0,+\infty)\\
g(x,0)=g_0(x), &~x\in\Omega.
               \end{array}
  \right. \label{4.5}\end{equation}
Here $H = \int_0^A\tilde{h}(a)da$.

The eradicability for \eqref{4.1} implies the nonnegative zero-stabilizability for \eqref{4.5}. However, the nonnegative zero-stabilizability for \eqref{4.5} implies that
$$\lambda_1^\omega \geq r^*.$$
This follows as in \cite{AAV} by using of the comparison results for the solutions to parabolic equations.

(ii) If $\lambda_1^\omega > r^*$, then we consider $u(x,t) = L$ a.e. in $\omega \times (0,+\infty)$. Using the comparison result for linear age-structured population dynamics (see \cite{anita}) we get that
\begin{equation}\label{4.6}
y(x,a,t) \leq \tilde{y}(x,a,t) \text{ a.e., }
\end{equation}
where $\tilde{y}$ is the solution to
$$\left\{ \begin {array}{ll}
\partial _ty(x,a,t)+\partial _ay(x,a,t)+\mu (a)y(x,a,t)&~ \\
 \ \ \ \ \ \ -d\Delta y(x,a,t)=-\chi_{\omega }(x)Ly(x,a,t),&~ x\in\Omega, \ a\in (0,A), \ t\in(0,+\infty) \\
\partial _{\nu }y(x,a,t)=0,&~ x\in\partial\Omega, \ a\in (0,A), \ t\in(0,+\infty)\\
y(x,0,t)=\int_0^A\beta (a)y(x,a,t)da,&~ x\in\Omega, \ t\in(0,+\infty)\\
y(x,a,0)=y_0(x,a),&~ x\in\Omega, \ a\in (0,A).
               \end{array}
  \right. $$
Let $h_0(a) = 1, \forall a \in [0,A]$, $g_0(x) = ||y_0||_{\infty}$ a.e. $x\in \Omega$.

Using again the comparison result for linear age-structured population dynamics we get that
\begin{equation}\label{4.7}
\tilde{y}(x,a,t) \leq h(a,t)g(x,t) \text{ a.e., }
\end{equation}
where $h$ is the solution to \eqref{4.3} and $g$ is the solution to \eqref{4.4} corresponding to $ \mathcal{M}\equiv 0$ and $u\equiv L$. Since $h(\cdot,t)\rightarrow \bar{h}$ in $L^{\infty}(0,A)$ as $t\rightarrow +\infty$ (\cite{anita}), and $g(\cdot,t)\rightarrow 0$ in $L^{\infty}(\Omega)$ as $t\rightarrow +\infty$ (because $\lambda_1^\omega > r^*$), we get by \eqref{4.6} and \eqref{4.7} that
$$\lim\limits_{t\to+\infty}y^L(\cdot,\cdot,t) = 0 \text{ in } L^{\infty}(\Omega\times(0,A)),$$
and the conclusion.
\end{proof}
Since our goal was actually to eradicate a pest population corresponding to a initial density $y_0$ with a harvesting effort less or equal than $L$ (tacking into account the above theorem) and since we have however to pay a certain cost to harvest in a subdomain $\omega$, we can consider the following related optimal control problem
$$\operatorname*{\mathit{Minimize}}_{\omega}\int_{0}^{A}\int_{\Omega}y(x,a,T)dx \ da + \alpha \ length(\partial\omega)+\beta \ area(\omega),
$$
where $T>0$ is a certain moment and $y$ is the solution to \eqref{4.1} corresponding to $u\equiv L$.

This problem may be investigated by using the level set method
described in  Section 2 and rewriting it in the following form
$$\operatorname*{\mathit{Minimize}}_{\varphi}\int_{0}^{A}\int_{\Omega}y_{\varphi}(x,a,T)dx \ da + \alpha\int_{\Omega} \delta (\varphi (x)) |\nabla \varphi(x)| dx+\beta\int_{\Omega} H(\varphi (x)) dx,
$$
where $y_{\varphi}$ is the solution to
$$\left\{ \begin {array}{ll}
\partial _ty(x,a,t)+\partial _ay(x,a,t)+\mu (a)y(x,a,t) -d\Delta y(x,a,t)=\\
 \ \ \ \ \ \ -\mathcal{M}\left(\int_0^Ay(x,a,t)da\right)y(x,a,t)\\
 \ \ \ \ \ \ -H(\varphi(x))Ly(x,a,t),~~~~~~~~~~~~~~x\in\Omega, \ a\in (0,A), \ t\in(0,+\infty) \\
\partial _{\nu }y(x,a,t)=0,~~~~~~~~~~~~~~~~~~~~~~~~~~~x\in\partial\Omega, \ a\in (0,A), \ t\in(0,+\infty)\\
y(x,0,t)=\int_0^A\beta (a)y(x,a,t)da, ~~~~~~~~x\in\Omega, \ t\in(0,+\infty)\\
y(x,a,0)=y_0(x,a), ~~~~~~~~~~~~~~~~~~~~~~x\in\Omega, \ a\in (0,A),
               \end{array}
  \right. $$
with $\varphi$ the implicit function of $\omega$. The approach is
similar to the one in Section 2. We will approximate this problem
using the mollified version of the Heaviside function,
$H_{\varepsilon}$, and its derivative by the mollified  function
$\delta_{\varepsilon}$.

Actually, if we denote by
$$\Psi(\varphi) = \operatorname*{\mathit{Minimize}}_{\varphi}\int_{0}^{A}\int_{\Omega}y_{\varphi}(x,a,T)dx \ da + \alpha\int_{\Omega} \delta_{\varepsilon} (\varphi (x)) |\nabla \varphi(x)| dx+\beta\int_{\Omega} H_{\varepsilon}(\varphi (x)) dx,$$
for a small but fixed $\varepsilon >0$, the harvesting problem to be investigated is
$$\underset{\varphi} {Minimize} \ \Psi(\varphi ), $$
where $\varphi:\overline{\Omega}\longrightarrow\mathbb{R}$ is a smooth function and $y_{\varphi}$ is the solution to
$$\left\{ \begin {array}{ll}
\partial _ty(x,a,t)+\partial _ay(x,a,t)+\mu (a)y(x,a,t) -d\Delta y(x,a,t)=\\
 \ \ \ \ \ \ -\mathcal{M}\left(\int_0^Ay(x,a,t)da\right)y(x,a,t)\\
 \ \ \ \ \ \ -H_{\varepsilon}(\varphi(x))Ly(x,a,t),~~~~~~~~~~~~~~x\in\Omega, \ a\in (0,A), \ t\in(0,+\infty) \\
\partial _{\nu }y(x,a,t)=0,~~~~~~~~~~~~~~~~~~~~~~~~~~~x\in\partial\Omega, \ a\in (0,A), \ t\in(0,+\infty)\\
y(x,0,t)=\int_0^A\beta (a)y(x,a,t)da, ~~~~~~~~x\in\Omega, \ t\in(0,+\infty)\\
y(x,a,0)=y_0(x,a), ~~~~~~~~~~~~~~~~~~~~~~x\in\Omega, \ a\in (0,A).
               \end{array}
  \right. $$
By  following the same lines as in Section 2 we can get the
directional derivative of $\Psi$. We reach a similar conclusion as
in Section 2 concerning the gradient descent with respect to
$\varphi $:
$$\left\{ \begin{array}{ll}
      \partial_{\theta } \varphi(x,\theta ) = \delta_{\varepsilon}(\varphi (x,\theta ))[-\alpha\; \text{\rm div}
      \left( \frac{\nabla \varphi(x,\theta )}{|\nabla \varphi(x, \theta )|}\right)+\beta&\\
 \ \ \ \ \ \ \  \ \ \ \ \ \ \     -L\int_0^A\int_0^T r(a,x,t)y_{\varphi}(a,x,t)da\ dt ], & x \in \Omega , \ \theta >0 \\
\frac{\delta_{\varepsilon}(\varphi(x,\theta ))}{|\nabla \varphi(x,\theta )|}\partial_{\nu} \varphi (x,\theta )=0, &  x\in \partial\Omega , \ \theta >0 .
           \end{array} \right. $$
where $\theta$ is an artificial time and
$r$ is solution to
$$\left\{ \begin {array}{ll}
\partial _tr(x,a,t)+\partial _ar(x,a,t)-\mu (a)r(x,a,t) +d\Delta r(x,a,t)=\\
 \ \ \ \ \ \ ~~\mathcal{M'}\left(\int_0^Ay(x,a,t)da\right)\int_0^Ar(x,a,t)y(x,a,t)da\\
 \ \ \ \ \ \ +\mathcal{M}\left(\int_0^Ay(x,a,t)da\right)r(x,a,t)+LH_{\varepsilon}(\varphi(x))r(x,a,t)\\
 \ \ \ \ \ \ -\beta(a)r(x,0,t),~~~~~~~~~~~~~~x\in\Omega, \ a\in (0,A), \ t\in(0,+\infty) \\
\partial _{\nu }r(x,a,t)=0,~~~~~~~~~~~~~~~~~~~~x\in\partial\Omega, \ a\in (0,A), \ t\in(0,+\infty)\\
r(x,A,t)=0, ~~~~~~~~~~~~~~~~~~~~~~x\in\Omega, \ t\in(0,+\infty)\\
r(x,a,T)=1, ~~~~~~~~~~~~~~~~~~~~~~x\in\Omega, \ a\in (0,A).
               \end{array}
  \right. $$
\section*{Acknowledgements}
This work was supported by the CNCS-UEFISCDI (Romanian National Authority for Scientific Research) grant
68/2.09.2013, PN-II-ID-PCE-2012-4-0270: ``Optimal Control and Stabilization of Nonlinear Parabolic Systems with State Constraints. Applications in Life Sciences and Economics''.

\end{document}